\documentclass{proc-l-hijacked}
\usepackage{amssymb, amsmath, amsthm, hyperref,}

\newtheorem{theorem}{Theorem}[section]

\theoremstyle{definition}

\newtheorem{example}[theorem]{Example}

\newtheorem{remark}[theorem]{Remarks}

\numberwithin{equation}{section}

\DeclareMathOperator{\Sa}{\mathcal S}
\DeclareMathOperator{\rad}{Rad}
\DeclareMathOperator{\Ima}{Im}
\DeclareMathOperator{\Rea}{Re}

\begin{document}
	\title[]{A spectral characterization of isomorphisms on $C^\star$-algebras}
	\author{ C. Tour\'e, F. Schulz and R. Brits}
	\address{Department of Mathematics, University of Johannesburg, South Africa}
	\email{cheickkader89@hotmail.com, francoiss@uj.ac.za, rbrits@uj.ac.za}
	
	\begin{abstract}
Following a result of Hatori, Miura and Tagaki \cite{unitalspec} we give here a spectral characterization of an isomorphism from a $C^\star$-algebra onto a Banach algebra. We then use this result to show that a $C^\star$-algebra $A$ is isomorphic to a Banach algebra $B$ if and only if there exists a surjective function $\phi:A\rightarrow B$ satisfying (i) $\sigma\left(\phi(x)\phi(y)\phi(z)\right)=\sigma\left(xyz\right)$  for all $x,y,z\in A$ (where $\sigma$ denotes the spectrum), and (ii) $\phi$ is continuous at $\mathbf 1$. A simple example shows that (i) cannot be relaxed to products of two elements, as is the case with commutative Banach algebras. Our results also elaborate on a paper (\cite{specproperties}) of Bre\v{s}ar and \v{S}penko.    
	\end{abstract}
	\parindent 0mm
	
	\maketitle

\section{Introduction}

In general $A$ will be a unital and complex Banach algebra, with the unit denoted by $\mathbf 1$. The invertible group of $A$ will be denoted by $G(A)$.  If $x\in A$ then the spectrum of $x$ (relative to $A$) is the (necessarily non-empty and compact) set $\sigma(x,A):=\{\lambda\in\mathbb C:\lambda\mathbf 1-x\notin G(A)\},$ and the spectral radius of $x\in A$ is defined as $\rho(x,A):=\sup\{|\lambda|:\lambda\in\sigma(x,A)\}.$ For $x\in A,$  $C_{\{x\}}$ denotes the bicommutant of the set $\{x\}$. Recall that if $y\in C_{\{x\}}$ then $\sigma(y,A)=\sigma\left(y,C_{\{x\}}\right)$, in which case we shall simply write $\sigma(y)$ for the spectrum of $y$, and $\rho(y)$ for the spectral radius of $y$; the same convention will be used whenever the algebra under consideration is clear from the context. We shall further write $\rad(A)$ for the radical of $A$, and $Z(A)$ for the centre of $A$, which is, by definition, equal to $C_{\{\mathbf 1\}}$.        

The proofs of the results in the current paper rely fundamentally on the following two theorems, and the well-known Lie-Trotter Formula (\cite[p.67]{aupetit1991primer}) for exponentials.  

\begin{theorem}[{\cite[Theorem 2.6]{univari}}]\label{unique} Let $A$ be a semisimple Banach algebra and $a,b\in A$. Then $a=b$ if and only if $\sigma(ax)=\sigma(bx)$ for all $x\in A$ satisfying $\rho(x-\mathbf 1)<1$. In particular $a=b$ if and only if $\sigma(ae^y)=\sigma(be^y)$ for all $y\in A$.
\end{theorem}
 
\begin{theorem}[{\cite[Corollary 3.3]{unitalspec}}]\label{commcase} Let $A$  be a unital, semisimple, and commutative Banach algebra, and let $B$ be a unital and commutative Banach algebra. Suppose that $\phi$ is a map from $A$ onto $B$ such that the equations
\begin{itemize}	
\item[(i)]{$\phi(\mathbf 1)=\mathbf 1,$}
\item[(ii)]{$\sigma(\phi(x)\phi(y))=\sigma(xy)$ for all $x,y\in A$}
\end{itemize}
hold. Then B is semisimple and $\phi$ is an isomorphism.
\end{theorem}

The current paper is in fact motivated by Theorem~\ref{commcase}, which is the main result of \cite{unitalspec}, as well as the results in Section 4.1 of
\cite{specproperties}. The following example shows that Theorem~\ref{commcase} fails in even the simplest of non-commutative cases:
\begin{example}\label{counter}
Let $A=B=M_2(\mathbb C)$ and define $\phi:A\rightarrow B$ by $\phi(a)=a^t$ where $a^t$ denotes the transpose of the matrix $a$. Then $\phi(\mathbf 1)=\mathbf 1,$ $\sigma(\phi(x)\phi(y))=\sigma(xy)$, for all $x,y\in A$, $\phi$ is surjective  (and $\phi$ is injective, linear and continuous). But clearly $\phi$ is not an isomorphism because $\phi$ is not multiplicative.   
\end{example} 
	 
In pursuance of our main results in Section 2, we need the following:

\begin{theorem}\label{properties}
Let $A$ be a semisimple Banach algebra, and let $\phi$ be a function from $A$ onto a Banach algebra $B$ satisfying
\begin{equation}\label{assumption}
\sigma\left(\prod_{i=1}^m\phi(x_i)\right)=\sigma\left(\prod_{i=1}^mx_i\right)
\end{equation}
 for all $x_i\in A$, $1\leq m\leq3$.
Then
\begin{itemize}
\item[(i)]{$\phi$ is bijective,}
\item[(ii)]{$B$ is semisimple,}
\item[(iii)]{$\phi$ is multiplicative i.e. $\phi(xy)=\phi(x)\phi(y)$ for all $x,y\in A,$}
\item[(iv)]{$\phi(\mathbf 1)=\mathbf 1$ and $\phi(\lambda x)=\lambda\phi(x)$ for all $x\in A$, $\lambda\in\mathbb C$, }
\item[(v)]{if $x,y\in A$ then $xy=yx\Leftrightarrow \phi(x)\phi(y)=\phi(y)\phi(x),$}
\item[(vi)]{if $x\in A$ then  $y\in C_{\{x\}}\Leftrightarrow \phi(y)\in C_{\{\phi(x)\}},$} 
\item[(vii)]{for each $x\in A$, $C_{\{x\}}\big/\rad\left(C_{\{x\}}\right)$ is isomorphic to $C_{\{\phi(x)\}}\big/\rad\left(C_{\{\phi(x)\}}\right).$} 	
\end{itemize}
\end{theorem}

\begin{proof}
(i) Suppose $\phi(x)=\phi(y)$. Then, using \eqref{assumption}, we have
\begin{equation*}
\sigma\left(\phi(x)\phi(z)\right)=\sigma\left(\phi(y)\phi(z)\right)\Rightarrow
\sigma\left(xz\right)=\sigma\left(yz\right)\mbox{ for all }z\in A.
\end{equation*}
 Since $A$ is semisimple Theorem~\ref{unique} gives $x=y$.\\
(ii) Let $\phi(x)\in\rad(B)$. Then, for any $z\in A$, we have
   	$$\sigma\left(xz\right)=\sigma\left(\phi(x)\phi(z)\right)=\{0\}.$$ 
   	Since $A$ is semisimple $x=0$. This shows that $\rad(B)$ is a singleton, which necessarily implies that $B$ is semisimple.\\
(iii) Fix $x,y\in A$. Then 
\begin{equation*}
\sigma\left(\phi(x)\phi(y)\phi(z)\right)=\sigma\left(xyz\right)=\sigma\left(\phi(xy)\phi(z)\right)\mbox{ for all }z\in A.
 \end{equation*}   	
 Since $\phi$ is surjective and $B$ is semisimple Theorem~\ref{unique} gives $\phi(xy)=\phi(x)\phi(y)$.\\
(iv) The first part follows from 
\begin{equation*}
\sigma(\phi(\mathbf 1)\phi(z))=\sigma(z)=\sigma(\mathbf 1\phi(z))\mbox{ for all }z\in A.
\end{equation*}
and the second part from
\begin{equation*}
\sigma(\phi(\lambda x)\phi(z))=\sigma(\lambda xz)=\sigma(\lambda\phi(x)\phi(z))\mbox{ for all }z\in A.
\end{equation*}
(v) If $xy=yx$ then 
 	\begin{equation*}
 	\sigma\left(\phi(x)\phi(y)\phi(z)\right)=\sigma\left(xyz\right)=\sigma\left(yxz\right)=\sigma\left(\phi(y)\phi(x)\phi(z)\right)\mbox{ for all }z\in A.
 	\end{equation*}
 So, as with (iii), we have $\phi(x)\phi(y)=\phi(y)\phi(x)$. The reverse implication is then obvious.\\
 (vi) Let $y\in C_{\{x\}}$, and suppose $\phi(a)\in B$ commutes with $\phi(x)$. By (v) $a$ commutes with $x$, and therefore with $y$. Again by (v) it follows that $\phi(y)$ commutes with $\phi(a)$ which proves the forward implication. The reverse implication holds similarly.\\
 (vii) Define 
 \begin{equation*}
 \widetilde{\phi}:C_{\{x\}}\big/\rad\left(C_{\{x\}}\right)\rightarrow C_{\{\phi(x)\}}\big/\rad\left(C_{\{\phi(x)\}}\right) 
 \end{equation*}
 by
 \begin{equation*}
 \widetilde\phi\left(w+\rad\left(C_{\{x\}}\right)\right)=\phi(w)+\rad\left(C_{\{\phi(x)\}}\right).
 \end{equation*}
 If $w,u\in C_{\{x\}}$ and $r\in\rad\left(C_{\{x\}}\right)$ with $w=u+r$ then, for each $y\in C_{\{x\}}$, 
 \begin{align*}
 wy=uy+ry&\Rightarrow\sigma(wy)=\sigma(uy+ry)\\&
 \Rightarrow\sigma(wy)=\sigma(uy)\\&
 \Rightarrow\sigma(\phi(w)\phi(y))=\sigma(\phi(u)\phi(y)).
 \end{align*}
 It follows from Theorem\ref{unique} that $\phi(w)-\phi(u)\in\rad\left(C_{\{\phi(x)\}}\right)$ and hence that $\widetilde\phi$ is well-defined. Since 
\begin{equation*}
\widetilde\phi\left(\mathbf 1+\rad\left(C_{\{x\}}\right)\right)=\mathbf 1+\rad\left(C_{\{\phi(x)\}}\right),
\end{equation*}
and
\begin{equation*}
\sigma\left(\widetilde{\phi}(a)\widetilde{\phi}(b)\right)=\sigma(ab)\mbox{ for all }a,b\in C_{\{x\}}\big/\rad\left(C_{\{x\}}\right)
\end{equation*}
Theorem~\ref{commcase} implies that $\widetilde{\phi}$ is an isomorphism.
\end{proof}

\bigskip

\section{A spectral characterization of isomorphisms on $C^\star$-algebras}

If $A$ is a $C^\star$-algebra, then we denote the real Banach space of self-adjoint elements of $A$ by 
$\Sa$. As usual we denote the real and imaginary parts of $x\in A$ by respectively 
$$\Rea x:=(x+x^\star)/2\mbox{ and }\Ima x:=(x-x^\star)/2i.$$

\begin{theorem}\label{main}
	Let $A$ be a $C^\star$-algebra, and $B$ be a Banach algebra. Then a surjective map $\phi:A\rightarrow B$ is an isomorphism from $A$ onto $B$ if and only if $\phi$ satisfies the following properties: 
	\begin{itemize}
	\item[(i)]{$\sigma\left(\prod_{i=1}^m\phi(x_i)\right)=\sigma\left(\prod_{i=1}^mx_i\right)$  for all $x_i\in A$, $1\leq m\leq3$.}
	\item[(ii)]{ $\phi$ is continuous at $\mathbf 1$.}	
	\end{itemize}
\end{theorem}
	
\begin{proof} If $\phi$ is an isomorphism from $A$ onto $B$, then (i) follows trivially, and (ii) follows from Johnson's Continuity Theorem \cite[Corollary 5.5.3]{aupetit1991primer}. We prove the reverse implication:
Let $a\in G(A)$. If $x_n\rightarrow a$, then $x_na^{-1}\rightarrow\mathbf 1$. By continuity of $\phi$ at $\mathbf 1$, together with Theorem~\ref{properties}(iii), it follows that
\begin{equation*}
\lim_n\phi(x_n)\phi\left(a^{-1}\right)=\lim\phi\left(x_na^{-1}\right)=\phi(\mathbf 1)=\mathbf 1.
\end{equation*}	
Hence $\lim_n\phi(x_n)=\phi(a)$, and so $\phi$ is continuous on $G(A).$ 
If $x\in A$ is normal, then, by the Fuglede-Putnam-Rosenblum Theorem \cite[Theorem 6.2.5]{aupetit1991primer}, $C_{\{x\}}$ is a commutative $C^*$-subalgebra of $A$, and hence semisimple. From Theorem~\ref{properties}(vi) it follows that the map 
\begin{equation*}
\widehat\phi:C_{\{x\}}\rightarrow C_{\{\phi(x)\}} 
\end{equation*}  
defined by $\widehat\phi(y)=\phi(y)$ for $y\in C_{\{x\}}$ is surjective, and it satifies \eqref{assumption}. So, by Theorems~\ref{properties} and~\ref{commcase}, we have that $\widehat\phi$ is an isomorphism; from this we may conclude that
\begin{equation}\label{normal}
 \phi\left(e^{x}\right)=e^{\phi(x)}\mbox{ if }x\mbox{ is normal. } 
 \end{equation}
 If $x,y\in\Sa$, then, using the Lie-Trotter Formula together with the continuity of $\phi$ on $G(A)$,  
\begin{align*}
e^{\phi(x+y)}&=\phi\left(e^{x+y}\right)=\phi\left(\lim_n\left(e^{\frac{x}{n}}e^{\frac{y}{n}}\right)^n\right)\\&=\lim_n\phi\left(\left(e^{\frac{x}{n}}e^{\frac{y}{n}}\right)^n\right)=\lim_n\left(e^{\frac{\phi(x)}{n}}e^{\frac{\phi(y)}{n}}\right)^n\\&=e^{\phi(x)+\phi(y)}
\end{align*}  
Replacing $x,y$ in the above equation by respectively $tx,ty$ where $t\in\mathbb R$ it follows that $e^{t\phi(x+y)}=e^{t\left(\phi(x)+\phi(y)\right)}$ from which differentiation with respect to $t$ gives $\phi(x+y)=\phi(x)+\phi(y)$. Thus, the restriction of $\phi$ to $\Sa$ is a linear map. Define now an auxiliary map
 $\psi_\phi:A\rightarrow B$ by 
 \begin{equation}\label{aux}
 \psi_\phi(x):=\phi\left(\Rea x\right)+i\phi\left(\Ima x\right).
 \end{equation}
Since the restriction of $\phi$ to the real Banach space $\mathcal S$ is linear, and $\mathbf 1\in\mathcal S$, continuity of $\phi$ at $\mathbf 1$ implies continuity of the restriction of $\phi$ to $\mathcal S.$ Therefore, since $x\mapsto\Rea x$ and $x\mapsto\Ima x$ are continuous on $A$, $\psi_\phi$ is continuous on $A$.  
Using the fact that $\phi$ is additive on $\Sa$, together with Theorem~\ref{properties}(iv), it is a  simple matter to show that $\psi_\phi$ is linear on $A$. Let $x=a+ib$ where $a=\Rea x$ and $b=\Ima x$. Then, since $a,b\in\Sa$, and since $\phi$ is multiplicative
 \begin{align*}
 \psi_\phi(x)^2&=\left(\phi(a)+i\phi(b)\right)^2\\&=
 \phi(a^2-b^2)+i\left[\phi(ab)+\phi(ba)\right],
 \end{align*}
 and
  \begin{align*}
  \psi_\phi(x^2)=
  \phi(a^2-b^2)+i\phi(ab+ba)
  \end{align*}
 But, since $\phi$ is additive on $\Sa$ and multiplicative
 \begin{equation*}
 \phi\left((a+b)^2\right)=\phi(a+b)^2=\left(\phi(a)+\phi(b)\right)^2
 \end{equation*} 
 from which expansion and comparison imply that 
 $\phi(ab)+\phi(ba)=\phi(ab+ba)$. Thus $\psi_\phi(x)^2=\psi_\phi(x^2)$ holds for each $x\in A$. By induction it then follows that $\psi_\phi(x)^{2^m}=\psi_\phi\left(x^{2^m}\right)$ holds for each $m\in\mathbb N$. From this we obtain
 \begin{align*}
 \psi_\phi\left(e^x\right)&=\psi_\phi\left(\lim_m\left(\mathbf 1+x/2^m\right)^{2^m}\right)=\lim_n\psi_\phi\left(\left(\mathbf 1+x/2^m\right)^{2^m}\right)\\&=
 \lim_n\psi_\phi\left(\mathbf 1+x/2^m\right)^{2^m}=\lim_n\left(\mathbf 1+\psi_\phi(x)/2^m\right)^{2^m}\\&=e^{\psi_\phi(x)}.
 \end{align*}
 But we also have 
 \begin{align*}
 \phi\left(e^x\right)&=\phi\left(e^{\Rea x+i\Ima x}\right)=\phi\left(\lim_n\left(e^{\frac{\Rea x}{n}}e^{\frac{i\Ima x}{n}}\right)^n\right)\\&=\lim_n\phi\left(\left(e^{\frac{\Rea x}{n}}e^{\frac{i\Ima x}{n}}\right)^n\right)=\lim_n\left(\phi\left(e^{\frac{\Rea x}{n}}\right)\phi\left(e^{\frac{i\Ima x}{n}}\right)\right)^n\\&=
 \lim_n\left(e^{\frac{\phi(\Rea x)}{n}}e^{\frac{i\phi(\Ima x)}{n}}\right)^n=e^{\phi\left(\Rea x\right)+i\phi\left(\Ima x\right)}\\&=e^{\psi_\phi(x)}
 \end{align*}
  Therefore
  \begin{align}\label{expeq}
  \phi\left(e^x\right)=e^{\psi_\phi(x)}= \psi_\phi\left(e^x\right)\mbox{ for all }x\in A.
  \end{align} 
  Now, if $x\in A$ is normal, then, by combining \eqref{normal} and 
  \eqref{expeq}, we obtain 
  \begin{equation*}
  e^{\lambda\phi(x)}=\phi\left(e^{\lambda x}\right)=e^{\lambda\psi_\phi(x)},\ \lambda\in\mathbb C
  \end{equation*}
  from which it follows that
  \begin{equation}\label{normal2}
  \phi(x)=\psi_\phi(x)\mbox{ if }x\mbox{ is normal. } 
  \end{equation}
  
   Let $x,y\in A$ be arbitrary. If $t>0$ is sufficiently small (which we keep fixed), then each of the sets $\sigma(\mathbf 1+tx)$, $\sigma(\mathbf 1+ty)$ and $\sigma\left((\mathbf 1+tx)(\mathbf 1+ty)\right)$ does not contain zero, and does not separate $0$ from infinity. So it follows from \cite[Theorem 3.3.6]{aupetit1991primer} that there exists
  $a,b,c\in A$ such that 
  \begin{equation*}
  \mathbf 1+tx=e^a,\ \mathbf 1+ty=e^b, \mbox{ and }(\mathbf 1+tx)(\mathbf 1+ty)=e^c
  \end{equation*}
  So by \eqref{expeq} it follows, on the one hand, that
  \begin{align*}
  \phi\left((\mathbf 1+tx)(\mathbf 1+ty)\right)&=\psi_\phi\left(\mathbf 1+t(x+y)+t^2xy\right)\\&=\mathbf 1+t(\psi_\phi(x)+\psi_\phi(y))+t^2\psi_\phi(xy),
  \end{align*} 
  and on the other hand, using the fact that $\phi$ is multiplicative together with \eqref{expeq}, that 
    \begin{align*}
    \phi\left((\mathbf 1+tx)(\mathbf 1+ty)\right)&=\phi(\mathbf 1+tx)\phi(\mathbf 1+ty)\\&=\psi_\phi(\mathbf 1+tx)\psi_\phi(\mathbf 1+ty)\\&=(\mathbf 1+t\psi_\phi(x))(\mathbf 1+t\psi_\phi(y))\\&=\mathbf 1+t(\psi_\phi(x)+\psi_\phi(y))+t^2\psi_\phi(x)\psi_\phi(y).
    \end{align*} 
 Comparison of the two expressions yields 
  $\psi_\phi(xy)=\psi_\phi(x)\psi_\phi(y),$ and so $\psi_\phi$ is linear and multiplicative. We proceed to show that $\psi_\phi$ is surjective: Let $a\in G(A)$. Then $a$ admits a polar decomposition $a=hu$ where $h\in\Sa$ and $u$ is unitary. So, since both $h$ and $u$ are normal, it follows from \eqref{normal2} that  
  \begin{equation}\label{inveq}
  \phi(a)=\phi(h)\phi(u)=\psi_\phi(h)\psi_\phi(u)=\psi_\phi(hu)=\psi_\phi(a),
  \end{equation}
  and so $\psi_\phi$ agrees with $\phi$ on $G(A)$. Notice further, since $\phi$ is a spectrum preserving multiplicative bijection from $A$ onto $B$, we have that $\phi\left(G(A)\right)=G(B)$. Let $b\in B$ be arbitrary, and fix $\lambda\in\mathbb C$ such that $|\lambda|>\rho(b)$. If we write
  $b=\lambda\mathbf 1+(b-\lambda\mathbf 1)$ and observe that both terms in the decomposition belong to $G(A),$ then it follows that $b-\lambda\mathbf 1=\phi(a)$ for some $a\in G(A)$ and  $\lambda\mathbf 1=\phi(\lambda\mathbf 1)$. Hence, from \eqref{inveq}, we have that
  \begin{equation*}
  b=\phi(\lambda\mathbf 1)+\phi(a)=\psi_\phi(\lambda\mathbf 1)+\psi_\phi(a)
  =\psi_\phi(\lambda\mathbf 1+a)
  \end{equation*}
  which shows that $\psi_\phi$ is surjective. So, since $\psi_\phi$ is now an isomorphism from $A$ onto $B$, to obtain the required result it suffices to show that $\psi_\phi$ agrees with $\phi$ everywhere on $A$: Fix $x\in A$ arbitrary. Then for each $y\in A$ 
  \begin{equation}\label{equality}
  \sigma(\phi(x)\phi(y))=\sigma(xy)=\sigma(\psi_\phi(x)\psi_\phi(y))
  \end{equation}
  In particular \eqref{equality} and \eqref{expeq} imply that for each $y\in A$ 
  \begin{equation*}
    \sigma\left(\phi(x)\psi_\phi\left(e^y\right)\right)=\sigma\left(\phi(x)\phi\left(e^y\right)\right)=\sigma\left(\psi_\phi(x)\psi_\phi\left(e^y\right)\right)
  \end{equation*}
  Since the exponentials in $B$ are precisely the images of the exponentials in $A$ under $\psi_\phi$ it follows from Theorem~\ref{unique} that $\phi(x)=\psi_\phi(x)$ and hence $\phi=\psi_\phi$.  
  
\end{proof}

\begin{theorem}\label{main2}
	Let $A$ be a $C^\star$-algebra, and $B$ be a Banach algebra. Then $A$ is isomorphic to $B$ if and only if there exists a surjective map $\phi:A\rightarrow B$ such that
	\begin{itemize}
		\item[(i)]{$\sigma\left(\phi(x)\phi(y)\phi(z)\right)=\sigma\left(xyz\right)$  for all $x,y,z\in A$.}
		\item[(ii)]{ $\phi$ is continuous at $\mathbf 1$.}	
	\end{itemize}
\end{theorem}
\begin{proof}
Let $\phi(x)\in\rad(B)$. Then, for each $z\in A$,
\begin{equation*}
\sigma\left(\phi(x)\phi(\mathbf 1)\phi(z)\right)=\{0\}\Rightarrow\sigma(xz)=\{0\}.
\end{equation*} 
Since $A$ is semisimple $x=0$ whence it follows that $\rad(B)$ is a singleton, and consequently that $B$ is semisimple. If $a,b\in A$ satisfy $\phi(a)\phi(b)=\phi(b)\phi(a)$, then, for each $x\in A$, we have	
\begin{equation*}
\sigma\left(\phi(a)\phi(b)\phi(x)\right)=\sigma(abx)=\sigma(bax)=\sigma\left(\phi(b)\phi(a)\phi(x)\right)
\end{equation*}
which, by Theorem~\ref{unique}, implies that $ab=ba$. Since $\phi$ is surjective a similar argument proves that if $a,b\in A$ satisfy $ab=ba$, then $\phi(a)\phi(b)=\phi(b)\phi(a)$. This shows that $\phi(\mathbf 1)\in Z(B)$. So, together with the fact that $\sigma\left(\phi(\mathbf 1)^3\right)=\{1\},$ it follows by the semisimplicity of $B$ that $\phi(\mathbf 1)^3=\mathbf 1$; to see this, observe that, for any quasinilpotent element $q\in B$, we have $$\sigma\left(\phi(\mathbf 1)^3+q\right)\subseteq\sigma\left(\phi(\mathbf 1)^3\right)+\sigma(q)=\{1\},$$ and then apply
\cite[Theorem 5.3.2]{aupetit1991primer} to conclude that $\phi(\mathbf 1)^3=\mathbf 1$. 
If we fix $x,y\in A$ arbitrary then, by the assumption (i),
\begin{equation*}
\sigma\left(\phi(x)\phi(y)\phi(z)\right)=\sigma\left(xyz\right)=\sigma\left(\phi(\mathbf 1)\phi(xy)\phi(z)\right)\mbox{ for all }z\in A.
\end{equation*} 
Since $\phi$ is surjective Theorem~\ref{unique} gives 
$\phi(\mathbf 1)^{-1}\phi(x)\phi(y)=\phi(xy)$ for all $x,y\in A$.
Define
\begin{equation}\label{isodef}
\psi(x):=\phi(\mathbf 1)^{-1}\phi(x),\ \ x\in A.
\end{equation}
 From \eqref{isodef}, using the fact that $\phi(\mathbf 1)^{-1}\in Z(B)$,  it follows that
 \begin{align*}
 \psi(x)\psi(y)&=\phi(\mathbf 1)^{-1}\phi(x)\phi(\mathbf 1)^{-1}\phi(y)=\phi(\mathbf 1)^{-2}\phi(x)\phi(y)\\&=\phi(\mathbf 1)^{-1}\phi(xy)= \psi(xy)
 \end{align*}
for all $x,y\in A$. Thus $\psi$ is multiplicative. Observe then that
\[\psi(x)=\phi(\mathbf 1)^3\psi(x)=\phi(\mathbf 1)^2\phi(x)\]   
whence it follows, from (i), that 
 \[\sigma(\psi(x))=\sigma\left(\phi(\mathbf 1)^2\phi(x)\right)=\sigma(x).\] 
 Thus, the hypothesis (i) in Theorem~\ref{main} holds for $\psi$ with $m=1$. For $m=2,3$ this condition is consequently satisfied since $\psi$ is multiplicative. Obviously, continuity of $\psi$ at $\mathbf 1$ follows from continuity of $\phi$ at $\mathbf 1$, and surjectivity of $\psi$ from surjectivity of $\phi$ together with invertibility of $\phi(\mathbf 1)$. So $\psi$ satisfies the hypothesis, and hence the conclusion, of Theorem~\ref{main}.       	
\end{proof}
 In Theorem~\ref{main2}, one cannot expect $\phi$ to be an isomorphism: 
\begin{example}
	Let $A=B=M_2(\mathbb C)$ and define $\phi:A\rightarrow B$ by $\phi(a)=(-1/2+i\sqrt3/2)a$. Then $\sigma(\phi(x)\phi(y)\phi(z))=\sigma(xyz)$, for all $x,y,z\in A$, $\phi$ is surjective and $\phi$ is continuous. But $\phi$ is not an isomorphism. It is also not necessarily true that $\phi$ must be a multiple of the identity map (if $A=B$); to see this, take the $C^\star$-algebra $A=B=M_2(\mathbb C)\oplus\mathbb C$ and define $\phi:A\rightarrow B$ by $\phi((a,b))=((-1/2+i\sqrt3/2)a,b)$.  
\end{example} 

However, as a simple consequence of Theorem~\ref{main2}, we can now improve on Theorem~\ref{main}. For example we only require the spectral assumption to hold for $m=2,3$: 

\begin{theorem}\label{main3}
	Let $A$ be a $C^\star$-algebra, and $B$ be a Banach algebra. Then a surjective map $\phi:A\rightarrow B$ is an isomorphism from $A$ onto $B$ if and only if $\phi$ satisfies the following properties: 
	\begin{itemize}
		\item[(i)]{$\sigma\left(\prod_{i=1}^m\phi(x_i)\right)=\sigma\left(\prod_{i=1}^mx_i\right)$  for all $x_i\in A$, $2\leq m\leq3$.}
		\item[(ii)]{ $\phi$ is continuous at $\mathbf 1$.}	
	\end{itemize}
\end{theorem}

\begin{proof}
With $m=3$ we have, as in the proof of Theorem~\ref{main2}, that $\psi:A\rightarrow B$ defined by $\psi(x):=\phi(\mathbf 1)^{-1}\phi(x),$ $x\in A$ is an isomorphism. With $m=2$ the Spectral Mapping Theorem says that each $\lambda\in\sigma\left(\phi(\mathbf 1)\right)$ is a cube root as well as a square root of $1$ from which it follows that $\sigma\left(\phi(\mathbf 1)\right)=\{1\}$. Since $\phi(\mathbf 1)\in Z(B),$ and $B$ is semisimple, $\phi(\mathbf 1)=\mathbf 1$ and hence $\phi$ is an isomorphism. 
 	
\end{proof}	

\begin{remark}
In Theorem~\ref{commcase}, removing the requirement $\phi(\mathbf 1)=\mathbf 1$ from the hypothesis  would still give the result that $B$ is semisimple, and that $A$ and $B$ are isomorphic, but with the isomorphism given by the map $\psi(x):=\phi(\mathbf 1)^{-1}\phi(x)=\phi(\mathbf 1)\phi(x)$; this can be shown using essentially the same arguments as in the proof of Theorem~\ref{main2} above.   	
\end{remark}

\begin{remark}In \cite{specproperties} Bre\v{s}ar and \v{S}penko 
consider the condition $\sigma\left(\phi(x)\phi(y)\phi(z)\right)=\sigma\left(xyz\right)$  for all $x,y,z\in A$ where $\phi$ is  map from a Banach algebra $A$ onto a semisimple Banach algebra $B$. One difficulty that arises here is that the spectral assumption on $\phi$ does not seem to imply the semisimplicity of $A$ (in contrast to the case where $A$ is the assumed semisimple Banach algebra, with semisimplicity of $B$ following via the spectral assumption).  By application of Theorem~\ref{unique} one can nevertheless improve \cite[Corollary 4.1]{specproperties} to hold for all $x,y$ rather than for all invertible $x,y$, and one can therefore omit the linearity assumption in \cite[Corollary 4.2]{specproperties} with the same conclusion. As a final remark, Example~\ref{counter} shows that the spectral assumption (i) in Theorems~\ref{main} and \ref{main2} cannot be relaxed to products of two elements which leaves open the question of whether the continuity assumption (ii) on $\phi$ is perhaps superfluous?      
\end{remark}

 \bibliographystyle{amsplain}
 \bibliography{Spectral}

\end{document}